\documentclass[submission]{eptcs}
\usepackage[dvipsnames]{xcolor} 

\hypersetup{colorlinks, urlcolor=RoyalBlue, linkcolor=Red, citecolor=Green} 
\usepackage[
  backend=bibtex,
  natbib=true, 
  style=numeric, 
  sorting=nyt, 
  maxbibnames=10, 
  giveninits=true,
  backref=true, 
  url=false,
  doi=true,
  isbn=false,
  arxiv=abs, 
  date=year, 
  abbreviate = true 
  ]{biblatex}
    \AtEveryBibitem{
    \clearfield{language}
    \iffieldundef{eprinttype}{}{
      \clearfield{url} 
      \clearfield{urldate} 
      \clearfield{urlyear}
      \clearfield{urlmonth}
    }
  }

\usepackage{graphicx} 
\usepackage[dvipsnames]{xcolor} 

  \usepackage{tikz} 
  \usetikzlibrary{
    cd, 
    petri, 
    backgrounds, 
    arrows, 
    positioning, 
    decorations.markings, 
    calc,  
    fit, 
    shapes.multipart
  }

\usepackage{moirai}
\usepackage{mathtools} 
\usepackage{amssymb} 

\usepackage{amsthm} 
\usepackage{stmaryrd} 

\usepackage{relsize} 
\usepackage{microtype} 
\usepackage{multicol} 
\usepackage{csquotes} 
\usepackage{xspace}
\usepackage{enumitem}
\usepackage{fontawesome} 

\usepackage{tabularx, cellspace} 
\usepackage[capitalize]{cleveref}

\usepackage{datetime}
\newcounter{theoremUnified} 
\numberwithin{theoremUnified}{section} 
\numberwithin{theoremUnified}{section} 

\newtheoremstyle{plainStyle} 
{2mm} 
{2mm} 
{} 
{} 
{\bfseries} 
{.} 
{.5em} 
{} 

\newtheoremstyle{italicStyle} 
{2mm} 
{2mm} 
{\itshape} 
{} 
{\bfseries} 
{.} 
{.5em} 
{} 

\def\backgrnd{black!10}	

\pgfdeclarelayer{edgelayer}
\pgfdeclarelayer{nodelayer}
\pgfdeclarelayer{background}
\pgfsetlayers{background, edgelayer,nodelayer,main}

\tikzstyle{place}=
[circle,thick,draw=blue!75,fill=blue!20,minimum size=6mm]
\tikzstyle{antiplace}=
[circle,thick,draw=red!75,fill=red!20,minimum size=6mm]
\tikzstyle{transition}=
[rectangle,thick,draw=black!75,fill=black!20,minimum size=4mm]
\tikzstyle{inarrow}=[->, >=stealth, shorten >=.03cm,line width=1.5]
\tikzstyle{antiarrow}=[<-, red!75,  >=stealth, shorten >=.03cm,line width=1.5]

\tikzset{
	pics/netA/.style args={#1/#2/#3/#4/#5/#6/#7}{code={

					\node [place,label=above:$p_1$, tokens={%
								#1
							}] (-pl_1) {};

					\node [transition,label=above:$t$, label=below:#5] (-tr_1) [right = of -pl_1] {};

					\node [place,label=above:$p_2$, tokens={%
								#2
							}] (-pl_2) [right = of -tr_1] {};

					\node [transition,label=left:$v$, label=above:#6] (-tr_2) [below = of -tr_1] {};
					\node [transition,label=below:$u$, label=above:#7] (-tr_3) [below = of -tr_2] {};

					\node [place,label=below:$p_3$, tokens={%
								#3
							}] (-pl_3) [left = of -tr_3] {};

					\node [place,label=below:$p_4$, tokens={%
								#4
							}] (-pl_4) [right = of -tr_3] {};

					\draw[inarrow] (-pl_1) -- (-tr_1);
					\draw[inarrow] (-tr_1) -- (-pl_2);
					\draw[inarrow] (-pl_2) -- (-tr_2);
					\draw[inarrow] (-tr_2) -- (-pl_3);
					\draw[inarrow] (-tr_2) -- (-pl_4);
					\draw[inarrow] (-pl_3) -- (-tr_3);
					\draw[inarrow] (-tr_3) -- (-pl_4);
				}}
}

\newcommand{\Msets}[1]{{#1}^{\oplus}} 

\usepackage{tensor}
\newcommand{\Str}[1]{{#1}^{\otimes}} 

\newcommand{\Cp}{\fatsemi} 

\newcommand{\Homtotal}[1]{\operatorname{Hom}_{\,#1}} 
\newcommand{\Id}[1]{\text{id}_{#1}} 

\newcommand{\Fwd}[1]{{#1}^+}
\newcommand{\Bwd}[1]{{#1}^-}
\newcommand{\FFwd}[1]{{#1}^{+\!\!+}}
\newcommand{\BBwd}[1]{{#1}^=}
\newcommand{\FBwd}[1]{{#1}^\pm}
\newcommand{\BFwd}[1]{{#1}^\mp}

\newcommand{\CategoryC}{\mathcal{C}}
\newcommand{\CategoryD}{\mathcal{D}}

\newcommand{\FCSMC}{\mathbf{FCSMC}} 
\newcommand{\FSSMC}{\mathbf{FSSMC}} 
\newcommand{\PetriS}[1]{\mathbf{Petri}^{#1}} 
\newcommand{\PetriCommBound}{\PetriS{\Span}_B} 
\newcommand{\PetriFreeBound}{\PetriS{\Span}_B} 
\newcommand{\PetriSpan}{\PetriS{\Span}} 

\newcommand{\Fun}[1]{{{#1}^\sharp}} 
\newcommand{\NetSem}[1]{\left( #1, \Fun{#1}\right)} 

\newcommand{\Semantics}{\mathcal{S}} 

\let\epsilon\varepsilon
\newcommand{\Free}[1]{\mathfrak{F}\left(#1\right)} 
\newcommand{\FreeB}[2][]{\mathfrak{F}^{#1}_B\left(#2\right)} 
\newcommand{\Comm}[1]{\mathfrak{C}\kern-.1em\left(#1\right)} 
\newcommand{\CommB}[1]{\mathfrak{C}_B\kern-.2em\left(#1\right)} 

\newcommand{\Grothendieck}[1]{\textstyle\int{#1}} 
\newcommand{\GrothendieckS}[1]{\Grothendieck{\Fun{#1}}} 

\newcommand{\harpvecsign}{\scriptscriptstyle\rightharpoonup}
\newcommand{\harpoonvec}[2]{%
	\ifx\displaystyle#1\doalign{$\harpvecsign$}{#1#2}\fi
	\ifx\textstyle#1\doalign{$\harpvecsign$}{#1#2}\fi
	\ifx\scriptstyle#1\doalign{\scalebox{.6}[.9]{$\harpvecsign$}}{#1#2}\fi
	\ifx\scriptscriptstyle#1\doalign{\scalebox{.5}[.8]{$\harpvecsign$}}{#1#2}\fi
}
\newcommand{\doalign}[2]{%
	{\vbox{\offinterlineskip\ialign{\hfil##\hfil\cr#1\cr$#2$\cr}}}%
}

\newcommand{\Span}{\mathbf{Span}} 
\newcommand{\Cat}{\mathbf{Cat}} 

\newcommand{\Suchthat}[2]{\left\{#1 \: \middle\vert \: #2\right\}} 

\usetikzlibrary{decorations.pathmorphing}

\tikzset{snek/.style={decorate, decoration={snake, amplitude=2pt}}}
\tikzset{ 
	oriented WD/.style={
			every to/.style={
					out=0,in=180,draw
				},
			label/.style={
					font=\everymath\expandafter{\the\everymath\scriptstyle},
					inner sep=0pt,
					node distance=2pt and -2pt
				},
			semithick,
			node distance=1 and 1,
			decoration={
					markings, mark=at position \stringdecpos with \stringdec
				},
			ar/.style={
					postaction={decorate}
				},
			execute at begin picture={
					\tikzset{
						x=\bbx, y=\bby,
						every fit/.style={
								inner xsep=\bbx, inner ysep=\bby
							}
					}
				}
		},
	string decoration/.store in=\stringdec,
	string decoration={
			\arrow{stealth};
		},
	string decoration pos/.store in=\stringdecpos,
	string decoration pos=.7,
	bbx/.store in=\bbx,
	bbx = 1.5cm,
	bby/.store in=\bby,
	bby = 1.5ex,
	bb port sep/.store in=\bbportsep,
	bb port sep=1.5,
	bb port length/.store in=\bbportlen,
	bb port length=4pt,
	bb penetrate/.store in=\bbpenetrate,
	bb penetrate=0,
	bb min width/.store in=\bbminwidth,
	bb min width=1cm,
	bb rounded corners/.store in=\bbcorners,
	bb rounded corners=2pt,
	bb small/.style={
			bb port sep=1,
			bb port length=2.5pt,
			bbx=.4cm, bb min width=.4cm,
			bby=.7ex
		},
	bb medium/.style={
			bb port sep=1,
			bb port length=2.5pt,
			bbx=.4cm,
			bb min width=.4cm,
			bby=.9ex
		},
	bb/.code 2 args={
			\pgfmathsetlengthmacro{\bbheight}{\bbportsep * (max(#1,#2)+1) * \bby}
			\pgfkeysalso{
				draw,
				minimum height=\bbheight,
				minimum width=\bbminwidth,
				outer sep=0pt,
				rounded corners=\bbcorners,
				thick,
				prefix after command={
						\pgfextra{\let\fixname\tikzlastnode}
					},
				append after command={
						\pgfextra{
							\draw
							\ifnum #1=0
								{}
							\else
								foreach \i in {1,...,#1} {
								($(\fixname.north west)!{\i/(#1+1)}!(\fixname.south west)$) +(-
								\bbportlen,0)
								coordinate (\fixname_in\i) -- +(\bbpenetrate,0) coordinate (\fixname_in\i')
								}
							\fi
							\ifnum
								#2=0{}
							\else
								foreach \i in {1,...,#2} {
								($(\fixname.north east)!{\i/(#2+1)}!(\fixname.south east)$) +(-
								\bbpenetrate,0)
								coordinate (\fixname_out\i') -- +(\bbportlen,0) coordinate (\fixname_out\i)
								}
							\fi;
						}
					}
			}
		},
	bb name/.style={
			append after command={
					\pgfextra{
						\node[anchor=north] at (\fixname.north) {#1}
						;}
				}
		}
}

\usepackage{adjustbox}

\def\lrc{
\begin{tikzpicture}[scale=.33]
	\draw[-] (0,0) -| (1,1);
\end{tikzpicture}
}

\newcommand{\pb}{\arrow[dr, phantom, "\lrc", very near start]}
\tikzset{commutative diagrams/.cd,arrow style=tikz,diagrams={>=stealth'}}

\usepackage[scaled=0.8]{FiraMono}

\newlength{\seplen}
\newlength{\sepwid}
\def\firstblank{\,\rule{\seplen}{\sepwid}\,}

\bibliography{Bibliography}

\newtheorem{notation}{Notation}{\rmfamily}{\rmfamily}
\newtheorem{definition}{Definition}{\rmfamily}{\rmfamily}
\newtheorem{proposition}{Proposition}{\rmfamily}{\rmfamily}
\newtheorem{remark}{Remark}{\rmfamily}{\rmfamily}
\newtheorem{theorem}{Theorem}{\rmfamily}{\rmfamily}
\newtheorem{example}{Example}{\rmfamily}{\rmfamily}
{\rmfamily}{\rmfamily}

\theoremstyle{plainStyle}
\begin{document}
\title{A Categorical Semantics for Bounded Petri Nets}

\author{Fabrizio Genovese 
 \email{0000-0001-7792-1375}
 \institute{University of Pisa \& Statebox}
 \email{fabrizio.romano.genovese@gmail.com}
 \and
 Fosco Loregian 
 \email{0000-0003-3052-465X}
 \institute{Tallinn University of Technology}
 \email{fosco.loregian@gmail.com}
 \and
  Daniele Palombi 
  \email{0000-0002-8107-5439}
  \institute{Sapienza University of Rome}
  \email{danielepalombi@protonmail.com}
}

\def\titlerunning{Bounded Petri Nets}
\def\authorrunning{Genovese, Loregian, Palombi}

%
\maketitle
\begin{abstract}
  We provide a categorical semantics for bounded
  Petri nets, both in the collective- and individual-token
  philosophies. In both cases, we describe the process of bounding
  a net both internally (just constructing new categories of executions
  of a net using comonads) and externally (using lax-monoidal-lax functors).
  Our external semantics is non-local, meaning that tokens
  are endowed with properties that say something about the global state
  of the net. We then prove, in both cases, that the internal and external
  constructions are equivalent, by using machinery built on top of the
  Grothendieck construction. The individual-token case is harder, as it requires
  a more explicit reliance on abstract methods.
\end{abstract}
\paragraph*{\bf Acknowledgements}
The first author was supported by the project MIUR PRIN 2017FTXR7S “IT-MaTTerS” and by the \href{https://gitcoin.co/grants/1086/independent-ethvestigator-program}{Independent Ethvestigator Program}.

The second author was supported by the ESF funded Estonian IT Academy research measure (project 2014-2020.4.05.19-0001).

\bigskip\noindent
A video presentation of this paper can be found on Youtube at \href{https://www.youtube.com/watch?v=mv4dyTNgz60}{mv4dyTNgz60}.

\section{Introduction}\label{sec: introduction}
A Petri net is a simple thing: it consists of \emph{places} -- depicted as circles -- and
\emph{transitions} -- depicted as squares and connected to places by inbound and outbound arrows.
Places are seen as holding resources, called \emph{tokens} and represented by black circles, 
while transitions are seen as processes that convert a pre-determined amount of some resource 
into a pre-determined amount of some other resource, as prescribed by the connecting arrows.
The action of a transition on tokens is called \emph{firing}. Sequences of firings are
called \emph{executions}, with an example depicted below.
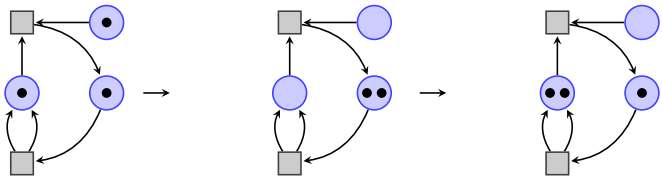
\begin{equation*}
  \scalebox{0.75}{
	\begin{tikzpicture}
		\begin{scope}
			\begin{pgfonlayer}{nodelayer}
				\node [place,tokens=1] (1a) at (1.5,1.25) {};
				\node [place,tokens=1] (1b) at (-0,0) {};
				\node [place,tokens=1] (3a) at (1.5,0) {};
				\node[transition] (2a) at (0,1.25) {};
				\node[transition] (2b) at (0,-1.25) {};
			\end{pgfonlayer}
			\begin{pgfonlayer}{edgelayer}
				\draw[style=inarrow, thick] (1a) to (2a);
				\draw[style=inarrow, thick] (1b) to (2a);
				\draw[style=inarrow, thick, bend right] (2b) to (1b);
				\draw[style=inarrow, thick, bend left] (2b) to (1b);

				\draw[style=inarrow, thick, bend left] (2a) to (3a);
				\draw[style=inarrow, thick, bend left] (3a) to (2b);
			\end{pgfonlayer}
		\end{scope}

		\begin{scope}[xshift=135]
			\begin{pgfonlayer}{nodelayer}
				\node [place,tokens=0] (1a) at (1.5,1.25) {};
				\node [place,tokens=0] (1b) at (-0,0) {};
				\node [place,tokens=2] (3a) at (1.5,0) {};
				\node[transition] (2a) at (0,1.25) {};
				\node[transition] (2b) at (0,-1.25) {};
			\end{pgfonlayer}
			\begin{pgfonlayer}{edgelayer}
				\draw[style=inarrow, thick] (1a) to (2a);
				\draw[style=inarrow, thick] (1b) to (2a);
				\draw[style=inarrow, thick, bend right] (2b) to (1b);
				\draw[style=inarrow, thick, bend left] (2b) to (1b);

				\draw[style=inarrow, thick, bend left] (2a) to (3a);
				\draw[style=inarrow, thick, bend left] (3a) to (2b);
			\end{pgfonlayer}
		\end{scope}

		\begin{scope}[xshift=270]
			\begin{pgfonlayer}{nodelayer}
				\node [place,tokens=0] (1a) at (1.5,1.25) {};
				\node [place,tokens=2] (1b) at (-0,0) {};
				\node [place,tokens=1] (3a) at (1.5,0) {};
				\node[transition] (2a) at (0,1.25) {};
				\node[transition] (2b) at (0,-1.25) {};
			\end{pgfonlayer}
			\begin{pgfonlayer}{edgelayer}
				\draw[style=inarrow, thick] (1a) to (2a);
				\draw[style=inarrow, thick] (1b) to (2a);
				\draw[style=inarrow, thick, bend right] (2b) to (1b);
				\draw[style=inarrow, thick, bend left] (2b) to (1b);

				\draw[style=inarrow, thick, bend left] (2a) to (3a);
				\draw[style=inarrow, thick, bend left] (3a) to (2b);
			\end{pgfonlayer}
		\end{scope}

		\draw[style=inarrow, thick] (2.15,0) -- (2.65,0);
		\draw[style=inarrow, thick] (7.05,0) -- (7.55,0);

	\end{tikzpicture}
}
\end{equation*}
This innocent definition~\cite{Petri2008} has proven very useful in domains ranging
from concurrency\cite{Jensen2009,Riemann1999},
where it is used to represent resource flow in distributed computing, to
chemistry~\cite{Reisig2013, Baez2017}, where
it models chemical reactions. In both cases, it becomes important to characterize which
sequences of transitions carry a given distribution of tokens for the net -- called \emph{marking} --
to some other marking. Here is where the similitude between applications specializes:
for instance, in chemistry, tokens represent molecules/atoms of a given type specified
by the place in which they live. Two tokens in the same place are physically indistinguishable, and
we do not care which token is consumed by a fixed transition. In computer
science applications -- e.g.~\cite{StateboxTeam2019} -- tokens are instead seen as terms of a given type,
and it is important to distinguish between them. These two philosophies are
called \emph{collective-token} and \emph{individual-token} philosophies, respectively~\cite{vanGlabbeek1999}.

In any case, one other important question is to establish when a given Petri net is \emph{bounded},
meaning that starting from a given marking, no place will hold more than a pre-determined number
of tokens throughout any possible firing~\cite{Reisig2013}. Traditionally, there is a simple way of turning any net
into a bounded one: we double-up the places adding what we call \emph{anti-places} (depicted in red), and edit
transitions so that each input (output) from (to) a place is now paired with a corresponding output
(input) to (from) the corresponding anti-place:
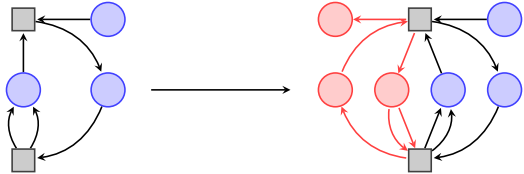
\begin{equation*}
  \scalebox{0.75}{
	\begin{tikzpicture}
		\begin{scope}[xshift=-100]
			\begin{pgfonlayer}{nodelayer}
				\node [place,tokens=0] (1a) at (1.5,1.25) {};
				\node [place,tokens=0] (1b) at (-0,0) {};
				\node [place,tokens=0] (3a) at (1.5,0) {};
				\node[transition] (2a) at (0,1.25) {};
				\node[transition] (2b) at (0,-1.25) {};
			\end{pgfonlayer}
			\begin{pgfonlayer}{edgelayer}
				\draw[style=inarrow, thick] (1a) to (2a);
				\draw[style=inarrow, thick] (1b) to (2a);
				\draw[style=inarrow, thick, bend right] (2b) to (1b);
				\draw[style=inarrow, thick, bend left] (2b) to (1b);

				\draw[style=inarrow, thick, bend left] (2a) to (3a);
				\draw[style=inarrow, thick, bend left] (3a) to (2b);
			\end{pgfonlayer}
		\end{scope}

		\begin{scope}[xshift=100]
			\begin{pgfonlayer}{nodelayer}
				\node [place,tokens=0] (1a) at (1.5,1.25) {};
				\node [place,tokens=0] (1b) at (0.5,0) {};
				\node [place,tokens=0] (3a) at (1.5,0) {};
				\node [antiplace,tokens=0] (1ax) at (-1.5,1.25) {};
				\node [antiplace,tokens=0] (1bx) at (-0.5,0) {};
				\node [antiplace,tokens=0] (3ax) at (-1.5,0) {};
				\node[transition] (2a) at (0,1.25) {};
				\node[transition] (2b) at (0,-1.25) {};
			\end{pgfonlayer}
			\begin{pgfonlayer}{edgelayer}
				\draw[style=inarrow, thick] (1a) to (2a);
				\draw[style=inarrow, thick] (1b) to (2a);
				\draw[style=inarrow, thick, bend right] (2b) to (1b);
				\draw[style=inarrow, thick] (2b) to (1b);

				\draw[style=inarrow, thick, bend left] (2a) to (3a);
				\draw[style=inarrow, thick, bend left] (3a) to (2b);

				\draw[style=antiarrow, thick] (1ax) to (2a);
				\draw[style=antiarrow, thick] (1bx) to (2a);
				\draw[style=antiarrow, thick] (2b) to (1bx);
				\draw[style=antiarrow, thick, bend left] (2b) to (1bx);

				\draw[style=antiarrow, thick, bend right] (2a) to (3ax);
				\draw[style=antiarrow, thick, bend right] (3ax) to (2b);
			\end{pgfonlayer}
		\end{scope}

		\draw[style=inarrow, thick] (-1.25,0) -- (1.25,0);
	\end{tikzpicture}
}
\end{equation*}
It is easy to see that for any marking, the amount of tokens in each anti-place will determine
the maximum number of tokens the corresponding place can hold through any execution, and vice-versa.

Building on the long-standing tradition of looking at Petri nets through a categorical lens~\cite{Meseguer1990},
we want to give a category-theoretic interpretation of this process, both in the collective-token philosophy
and in the individual-token one. In each case, we will provide two different categorical ways
to bound a net, which we will call \emph{internal} and \emph{external}, respectively.
Then, we will prove that internal and external bounding are conceptually the same thing,
generalizing a technique already established in~\cite{Genovese2020}. Notably, in doing so, we will employ a new
kind of categorical semantics for Petri nets, which we call \emph{non-local}~\cite{Genovese2021a}, where tokens
come endowed with information concerning the global properties of the net itself.

\section{Petri nets and their executions}
\begin{notation}
  Let $S$ be a set; denote with $\Msets{S}$ the set of multisets over $S$.
  Multiset sum will be denoted with $\oplus$, and difference (only partially defined) with
  $\ominus$. $\Msets{S}$ with $\oplus$ and the empty multiset is isomorphic to the free commutative monoid on $S$.
\end{notation}
\begin{definition}[Petri net]\label{def: Petri net}
  A \emph{Petri net} is a couple functions $T \xrightarrow{s,t} \Msets{S}$ for
  some sets $T$ and $S$, called the set of places and transitions of the net, respectively.
\end{definition}
\begin{definition}[Markings and firings]\label{def: Petri net firing}
  A \emph{marking} for a net  $T \xrightarrow{s,t} \Msets{S}$ is an element of $\Msets{S}$,
  representing a distribution of tokens in the net places. A transition $u$ is \emph{enabled} in a marking $M$ if
  $M \ominus s(u)$ is defined. An enabled transition can \emph{fire}, moving tokens in the net.
  Firing is considered an atomic event, and the marking resulting from firing $u$ in $M$ is $M \ominus s(u) \oplus t(u)$.
\end{definition}
As we said, depending on whether we want to consider tokens as indistinguishable or not, we can interpret
nets under two different philosophies, traditionally called collective-token and individual-token, respectively.
Category theory is helpful in pinpointing precisely the meaning of this distinction by providing
different notions for the \emph{category of executions of a given net}. In the collective-token philosophy,
the executions of a Petri net are taken to be morphisms in a \emph{commutative monoidal category} -- a category
whose monoid of objects is commutative:
\begin{definition}[Category of executions -- collective-token philosophy]\label{def: executions common token philosophy}
  Let $N: T \xrightarrow{s,t} \Msets{S}$ be a Petri net.
  We can generate a \emph{free commutative strict monoidal category (FCSMC)}, $\Comm{N}$, as follows:
  \begin{itemize}
    \item The monoid of objects is $\Msets{S}$.
    \item Morphisms are generated by $T$: each $u \in T$ corresponds to a morphism generator $su \xrightarrow{u} tu$; morphisms are obtained
          by considering all the formal (monoidal) compositions of generators and identities.
  \end{itemize}
  A detailed description of this construction can be found in~\cite{Master2020}.
\end{definition}
\begin{example}
  In~\cref{def: executions common token philosophy} objects represent markings of a net: $A\oplus A \oplus B$ means ``two tokens
  in $A$ and one token in $B$''. Morphisms represent executions of a net, mapping markings to markings.
  A marking is reachable from another one if and only if there is a morphism between them.
  \begin{equation*}
    \resizebox{\textwidth}{9em}
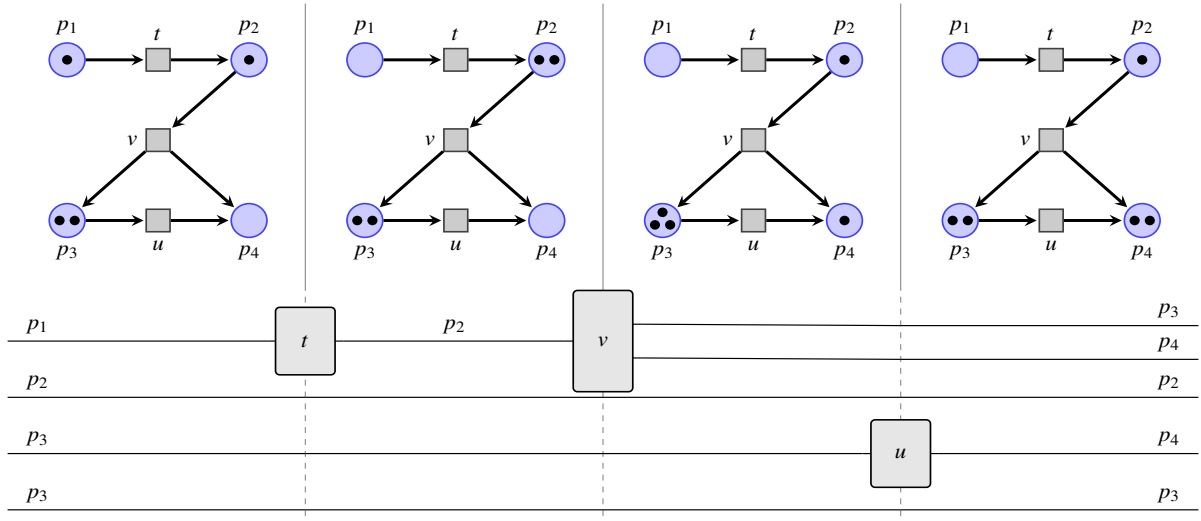
{
	\begin{tikzpicture}
		\pgfmathsetmacro\bS{5}
		\pgfmathsetmacro\hkX{(\bS/3.5)}
		\pgfmathsetmacro\kY{-1.5}
		\pgfmathsetmacro\hkY{\kY*0.5}
		\draw pic (m0) at (0,0) {netA={{1}/{1}/{2}/{0}/{}/{}/{}}};
		\draw pic (m1) at (\bS,0) {netA={{0}/{2}/{2}/{0}/{}/{}/{}}};
		\draw pic (m2) at ({2 * \bS},0) {netA={{0}/{1}/{3}/{1}/{}/{}/{}}};
		\draw pic (m3) at ({3 * \bS},0) {netA={{0}/{1}/{2}/{2}/{}/{}/{}}};
		\begin{scope}[very thin]
			\foreach \j in {1,...,3} {
					\pgfmathsetmacro \k { \j * \bS - 1 };
					\draw[gray,dashed] (\k,-4) -- (\k,-8.25);
					\draw[gray] (\k,1) -- (\k,-4);
				}
		\end{scope}
		\begin{scope}[shift={(0,-4)}, oriented WD, bbx = 1cm, bby =.4cm, bb min width=1cm, bb port sep=1.5]
			\draw node [fill=\backgrnd,bb={1}{1}] (Tau) at (\bS -1,-1) {$t$};
			\draw node [fill=\backgrnd,bb={1}{2}, ] (Mu)  at ({2 * \bS - 1},-1) {$v$};
			\draw node [fill=\backgrnd,bb={1}{1}] (Nu)  at ({3 * \bS - 1},{2 * \kY}) {$u$};
			\draw (-1,-1) --     node[above] {$p_1$}       (0,-1)
			--                  node[above] {}          (Tau_in1);
			\draw (-1,-2) -- node[above] {$p_2$} (0,-2) -- (\bS-1, -2);
			\draw (-1,-3) -- node[above] {$p_3$} (0,-3) -- (\bS-1, -3);
			\draw (-1,-4) -- node[above] {$p_3$} (0,-4) -- (\bS-1, -4);
			\draw (Tau_out1) -- node[above] {$p_2$}    (Mu_in1);
			\draw (\bS-1,-2) -- (2*\bS-1, -2);
			\draw (\bS-1,-3) -- (2*\bS-1, -3);
			\draw (\bS-1,-4) -- (2*\bS-1, -4);
			\draw (Mu_out1) --  (3*\bS-1, -0.725);
			\draw (Mu_out2) --  (3*\bS-1, -1.325);
			\draw (2*\bS-1,-2) -- (3*\bS-1, -2);
			\draw (2*\bS-1,-3) -- (Nu_in1);
			\draw (2*\bS-1,-4) -- (3*\bS-1, -4);
			\draw (3*\bS-1,-0.725) to (4*\bS-2, -0.725) -- node[above] {$p_3$} (4*\bS-1, -0.725);
			\draw (3*\bS-1,-1.325) -- (3*\bS,-1.325) to (4*\bS-2, -1.325) -- node[above] {$p_4$} (4*\bS-1, -1.325);
			\draw (3*\bS-1,-2) to (4*\bS-2, -2) -- node[above] {$p_2$} (4*\bS-1, -2);
			\draw (Nu_out1) to (4*\bS-2, -3) -- node[above] {$p_4$} (4*\bS-1, -3);
			\draw (3*\bS-1,-4) to (4*\bS-2, -4) -- node[above] {$p_3$} (4*\bS-1, -4);
		\end{scope}
	\end{tikzpicture}
}
  \end{equation*}
\end{example}
As for the individual-token philosophy, we obtain a suitable semantics
by relaxing the commutativity requirement in~\cref{def: executions common token philosophy}.
\begin{definition}[Category of executions -- individual-token philosophy]\label{def: executions individual token philosophy}
  Let $N: T \xrightarrow{s,t} \Msets{S}$ be a Petri net. We can generate
  a \emph{free symmetric strict monoidal category (FSSMC)}, $\Free{N}$, as follows:
  \begin{itemize}
    \item The monoid of objects is $\Str{S}$, the set of strings over $S$. Monoidal product of objects $A,B$,
          denoted $A \otimes B$, is given by string concatenation.
    \item Morphisms are generated by $T$:  each $u \in T$ corresponds to a morphism generator $su \xrightarrow{u} tu$,
          where $su, tu$ are obtained by choosing some ordering on their underlying multisets; morphisms are obtained by considering all the formal horizontal and vertical compositions of generators, identities and symmetries.
  \end{itemize}
  A detailed description of this construction can be found in~\cite{Genovese2019c}.
\end{definition}
\begin{example}\label{ex: net executions difference in causality}
  The interpretation of~\cref{def: executions individual token philosophy}
  is as in~\cref{def: executions common token philosophy}, but
  now switching tokens around is not anymore a trivial operation. For instance,
  the following morphisms are considered different in~\cref{def: executions individual token philosophy},
  and equal in~\cref{def: executions common token philosophy}:
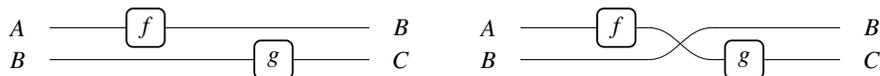
  \begin{equation*}
    \scalebox{0.85}{
	\begin{tikzpicture}[baseline=0.45cm]
		\node (A) at (-0.5,0.5) {$A$};
		\node (B) at (-0.5,0) {$B$};
		\node (B') at (5.5,0.5) {$B$};
		\node (C) at (5.5,0) {$C$};

		\node[draw,thick,minimum width=0.6cm,minimum height=0.6cm, rounded corners=3pt] (f) at (1.5,0.5) {$f$};
		\node[draw,thick,minimum width=0.6cm,minimum height=0.6cm, rounded corners=3pt] (g) at (3.5,0) {$g$};
		\draw (0,0.5) -- (f);
		\draw (f) -- (5,0.5);
		\draw (0,0) -- (g);
		\draw (g) -- (5,0);
	\end{tikzpicture}
	\qquad
	\begin{tikzpicture}[baseline=0.45cm]
		\node (A) at (-0.5,0.5) {$A$};
		\node (B) at (-0.5,0) {$B$};
		\node (B') at (5.5,0.5) {$B$};
		\node (C) at (5.5,0) {$C$};

		\node[draw,thick,minimum width=0.6cm,minimum height=0.6cm, rounded corners=3pt] (f) at (1.5,0.5) {$f$};
		\node[draw,thick,minimum width=0.6cm,minimum height=0.6cm, rounded corners=3pt] (g) at (3.5,0) {$g$};
		\draw (0,0.5) -- (f);
		\draw (f) -- (2,0.5);
		\draw (0,0) -- (2,0);

		\draw[out=0, in=180] (2,0.5) to (3,0);
		\draw[out=0, in=180] (2,0) to (3,0.5);

		\draw(3,0.5) -- (5,0.5);
		\draw(3,0) -- (g);
		\draw (g) -- (5,0);
	\end{tikzpicture}}
  \end{equation*}
\end{example}
A great deal of work has been devoted to understanding how categories of Petri nets
and categories of (commutative, symmetric) strict monoidal categories are
related~\cite{Meseguer1990, Sassone1995,Genovese2019b, Baez2020a, Baldan2009, ,Bruni2013, Baez2021}.
We will not focus on
these issues in this paper; instead, we will be interested in recovering the idea of ``bounding a net'' as a functor between the category of executions of a net and some other category.

The approach we pursue will be very similar to the one we already adopted for coloured nets in~\cite{Genovese2020},
with the difference that the functors here will have to be lax-monoidal-lax, something we already
partially exploited in~\cite{Genovese2021a}. For obvious reasons,
we start from the collective-token case, which is simpler.
\section{Bound semantics in the collective-token philosophy}\label{sec: bound semantics collective}
We briefly described the process of turning a Petri net into a bounded one in~\cref{sec: introduction}.
Categorically, this process can be implemented in the executions semantics,
as follows.
\begin{definition}[Internal bound semantics -- collective-token philosphy]\label{def: internal bound semantics common token philosophy}
  Let $N$ be a Petri net, and consider $\Comm{N}$, its corresponding FCSMC.
  The \emph{internal bound semantics of $N$ in the collective-token philosophy}
  is given by the FCSMC $\CommB{N}$ generated as follows:
  \begin{itemize}
    \item For each generating object $A$ in $\Comm{N}$, we introduce object generators $\Fwd{A}$, $\Bwd{A}$.
    \item For each generating morphism
          \begin{equation*}
            \textstyle\bigoplus_{i=1}^n A_i \xrightarrow{u} \bigoplus_{j=1}^m B_j
          \end{equation*}
          in $\Comm{N}$, we introduce a morphism generator:
          \begin{equation*}
            \textstyle\bigoplus_{i=1}^{m\vee n} \left(\Fwd{A}_i \oplus \Bwd{B_i} \right)
            \xrightarrow{u}
            \bigoplus_{i=1}^{m\vee n} \left(\Bwd{A}_i \oplus \Fwd{B_i} \right)
          \end{equation*}
          where we adopt the convention that, if $n < m$ (respectively $n > m$), then $A_{n+1}, \dots, A_m$
          (respectively $B_{m_1}, \dots, B_n$) are taken to be equal to $I$, the monoidal unit, and $m\vee n:=\max\{m,n\}$.
  \end{itemize}
\end{definition}
\begin{example}
  Notice that by definition, every FCSMC is presented by a Petri net.
  Following~\cref{def: internal bound semantics common token philosophy},
  it is easy to see that if we consider $N$ to be the net in~\cref{sec: introduction},
  then $\CommB{N}$ is presented exactly by the bounded net we na\"ively associated to it.
  Indeed, generating objects of type $\Fwd{A}$ represent a token in the place $A$ of the net while generating objects of type $\Bwd{A}$ represent tokens in the corresponding anti-place.
  Generating morphisms of $\CommB{N}$ have inputs and outputs exactly as the transitions in the net obtained by bounding $N$.
\end{example}
\begin{proposition}\label{prop: comonad collective-token philosophy}
  The assignment $\Comm{N} \mapsto \CommB{N}$ defines a comonad in
  the category of FCMSCs and strict monoidal functors between them, $\FCSMC$.
\end{proposition}
\begin{proof}
  First of all, we have to prove that the procedure is functorial. For any strict monoidal functor
  $F : \Comm{N} \to \Comm{M}$ we define the action on morphisms $\CommB{F}: \CommB{N} \to \CommB{M}$
  as the following monoidal functor:
  \begin{itemize}
    \item If for a generating object $A$ of $\Comm{N}$ it is $FA = B$, then $\CommB{F}\Fwd{A} = \Fwd{B}$ and $\CommB{F}\Bwd{A} = \Bwd{B}$.
    \item If for a generating morphism $t$ of $\Comm{N}$ it is $Ft = u$, then $\CommB{F}t = u$.
  \end{itemize}
  Identities and compositions are clearly respected, making $\CommB{\firstblank}$ an endofunctor
  in $\FCSMC$.
  As a counit, on each component $N$ we define the strict monoidal functor
  $\epsilon_N: \CommB{N} \to \Comm{N}$ sending:
  \begin{itemize}
    \item The generating object $\Fwd{A}$ to $A$, and the generating object $\Bwd{A}$ to $I$, the monoidal unit of $\Comm{N}$.
    \item A generating morphism
          \begin{equation*}
            \textstyle\bigoplus_{i=1}^{m\vee n} \left(\Fwd{A}_i \oplus \Bwd{B_i} \right)
            \xrightarrow{u}
            \bigoplus_{i=1}^{m\vee n} \left(\Bwd{A}_i \oplus \Fwd{B_i} \right)
          \end{equation*}
          to the generating morphism
          \begin{equation*}
            \textstyle\bigoplus_{i=1}^n A_i \xrightarrow{u} \bigoplus_{j=1}^m B_j.
          \end{equation*}
        \end{itemize}
        The procedure is natural in the choice of $N$, making $\epsilon$ into a natural transformation $\CommB{\firstblank} \to \Id{\FCSMC}$.

        As for the comultiplication, on each component $N$ we define the strict monoidal functor
        $\delta_N: \CommB{N} \to \CommB{\CommB{N}}$ sending:
        \begin{itemize}
          \item the generating object $\Fwd{A}$ to the generating object $\FFwd{A} \oplus \BBwd{A}$;
          \item the generating object $\Bwd{A}$ to the generating object $\BFwd{A} \oplus \FBwd{A}$;
          \item a generating morphism of $\CommB{N}$
                \begin{equation*}
                  \textstyle\bigoplus_{i=1}^{m\vee n} \left(\Fwd{A}_i \oplus \Bwd{B_i} \right)
                  \xrightarrow{u}
                  \bigoplus_{i=1}^{m\vee n} \left(\Bwd{A}_i \oplus \Fwd{B_i} \right)
                \end{equation*}
                to the generating morphism $\CommB{\CommB{N}}$
                \begin{align*}
                  \textstyle\bigoplus_{i=1}^{m\vee n} \left(\FFwd{A}_i \oplus \BBwd{A}_i \oplus \FBwd{B}_i \oplus \BFwd{B}_i \right)
                  \xrightarrow{u}
                  \bigoplus_{i=1}^{m\vee n} \left(\BFwd{A}_i \oplus \FBwd{A}_i \oplus \BBwd{B}_i \oplus \FFwd{B}_i \right)
                \end{align*}
                The naturality of $\delta$, its coassociativity and the fact that $\epsilon$ is a counit are all straightforward to check. \qedhere
        \end{itemize}
\end{proof}
In the spirit of~\cite{Genovese2020}, our goal is now to frame this comonad into a bigger context, one of Petri nets with a \emph{semantics} attached to them. A
semantics for a Petri net is just a functor from its category of executions
(modulo choice of token philosophy) to some other monoidal category $\Semantics$.

In~\cite{Genovese2020}, this functor was required to be strict monoidal. This backed up the interpretation
that a semantics ``attaches extra information to tokens'', which is then used by the transitions
in a suitable way.

Here, we require this functor to be \emph{lax-monoidal-lax}. Laxity amounts to saying
that we can attach \emph{non-local} information to tokens: they may ``know'' something
about the overall state of the net, and the laxator represents the process of ``tokens joining knowledge''.
If this sounds handwavy, we guarantee it will be made clear shortly.
\begin{definition}[Non-local semantics -- collective-token philosophy]\label{def: non-local semantics collective-token philosphy}
  Let $N$ be a Petri net and let $\Semantics$ be a monoidal bicategory~\cite{Schommer-Pries2014}.
  A \emph{Petri net with a commutative non-local semantics} is a couple $\NetSem{N}$ with
  $\Fun{N}$ a \emph{lax-monoidal-lax functor} $\Comm{N} \to \Semantics$.

  A morphism $\NetSem{N} \to \NetSem{M}$ of Petri nets
  with commutative semantics is a strict monoidal functor
  $\Comm{N} \xrightarrow{F} \Comm{M}$ making the obvious triangle commute.
  We denote the category of Petri nets with a non-local commutative semantics with $\PetriS{\Semantics}$.
\end{definition}
We now show how we can encode the information
of a net being bounded as some particular kind of non-local semantics.
\begin{notation}
  Let $\Span$ be the monoidal bicategory of sets, spans and span morphisms.
  A morphism $A \to B$ in $\Span$ consists of a set $S$ and a pair of functions
  $A \leftarrow S \rightarrow B$. When we need to notationally extract this information from
  $f$, we write $A \xleftarrow{f_1} S_f \xrightarrow{f_2} B$.
  We sometimes consider a span as a morphism $f: S_f \to A \times B$, thus we may
  write $f(s)  = (a,b)$ for $s \in S_f$ with $f_1(s) = a$ and $f_2(s) = b$.
  Recall moreover that a 2-cell in $\Span$ $f \Rightarrow g$ is a function $\theta:S_f \to S_g$
  such that $f = \theta \Cp g$.
\end{notation}
\begin{remark}
  In a free commutative (resp.\@\xspace symmetric) strict monoidal category, there are no equations between generators. As such, even if a morphism can be
  decomposed in a composition of tensors of generators (and symmetries) in multiple ways,
  the generators used -- as well as how many times they are used -- is an invariant of this decomposition.
  Given this, there exists a well-defined function
  \begin{equation*}
    \begin{tikzcd}
    \chi : \Homtotal{\Free{N}} \ar[r] &
    \left\{
      \begin{smallmatrix}
      \text{generator mor-}\\
      \text{phisms of } \Msets{{\Free{N}}}
  \end{smallmatrix}
      \right\}
\end{tikzcd}
  \end{equation*}
  mapping each morphism $f$ to a multiset counting how many times each morphism generator is used in the decomposition of $f$.
\end{remark}
\begin{definition}[External bound semantics -- collective-token philosphy]\label{def: external bound semantics common token philosophy}
  Given a Petri net $N: T \xrightarrow{s,t} \Msets{S}$, define
  the following functor $\Fun{N}: \Comm{N} \to \Span$, which is lax-monoidal-lax (see \autoref{is_lax_m_lax}):
  \begin{itemize}
    \item Each object $A$ of $\Comm{N}$ is mapped to the set $\Msets{S}$, the set of objects of $\Comm{N}$.
    \item Each morphism $A \xrightarrow{f} B$ is sent to the span:
    \begin{equation*}
     \Msets{S} \xleftarrow{t} \chi(f)^{-1} \xrightarrow{s} \Msets{S}
    \end{equation*}
    with $s,t$ denoting source and target, respectively, of a morphism in $\Comm{N}$.
  \end{itemize}
\end{definition}
\begin{proposition}\label{is_lax_m_lax}
  The functor of~\cref{def: external bound semantics common token philosophy} is lax-monoidal-lax.
  The class of $\Fun{N}$ for all Petri nets $N$ as in~\cref{def: external bound semantics common token philosophy} form a subcategory of
  $\PetriSpan$, which we call $\PetriCommBound$.
\end{proposition}
\begin{proof}
  Since $\chi(\Id{A})$ is the empty multiset, and identities and symmetries coincide,
  $\chi(\Id{A})^{-1}$ coincides with the set of objects of $\Comm{N}$, and identities are
  preserved strictly.
  As for composition, the laxating 2-cell is given by the obvious inclusion
  $\chi(f)^{-1} \times_{\Msets{S}} \chi(g)^{-1} \to \chi(f \Cp g)^{-1}$
  obtained by composition. A very similar argument
  holds for the monoidal product noticing that $\chi(f \Cp g) = \chi(f \oplus g)$.
  The coherence conditions are tedious to check but straightforward, given that the structure of the laxators is very simple.\footnote{
    If this seems strange, recall Example~\ref{ex: net executions difference in causality}.
    As objects commute, morphisms in $\Comm{C}$ obey a very weak causal flow,
    and composition and monoidal product are quite similar.
  }
\end{proof}
\begin{example}
  We now try to shed some light on~\cref{def: external bound semantics common token philosophy}.
  Consider the net below. It has three places, which according to \cref{def: external bound semantics common token philosophy} are all sent to the set of multisets over themselves. We interpret tokens as endowed with elements of such a set.
  These represent markings on the anti-places of the net. In our example below, the token in $p_1$
  ``knows'' that there is one token in the anti-place $\Bwd{p}_1$ and four tokens in the anti-place
  $\Bwd{p}_3$, while $p_2$ ``knows'' that there are three tokens in $\Bwd{p}_1$ and $\Bwd{p}_2$, and two in $\Bwd{p}_3$.
  The laxator allows us to consider these two tokens as a unique entity, with their respective ``local pieces of knowledge'' summing up.
  \begin{equation*}
    \scalebox{0.7}
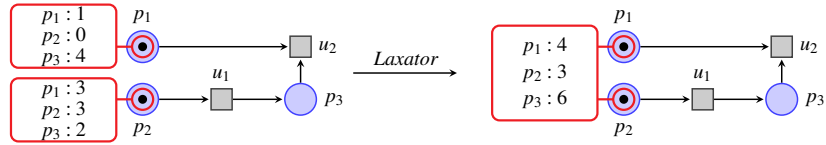
{  
\begin{tikzpicture}
    \begin{scope}[xshift=0]
    \begin{pgfonlayer}{nodelayer}
      \node [place,tokens=1, label=above:$p_1$] (1a) at (-1.5,1) {};      			
      \node [place,tokens=1, label=below:$p_2$] (1b) at (-1.5,0) {};
      \node [place,tokens=0, label=right:$p_3$] (3a) at (1.5,0) {};

      \node[transition, label=above:$u_1$] (2a) at (0,0) {};
      \node[transition, label=right:$u_2$] (4a) at (1.5,1) {};

      \draw[very thick, Red] (-1.5,0) circle (0.2);
        \draw[very thick, Red] (-1.7,0) -- (-2,0);
      \draw[very thick, Red] (-1.5,1) circle (0.2);
      \draw[very thick, Red] (-1.7,1) -- (-2,1);

      \draw[very thick, rounded corners, Red] (-4,1.8) rectangle (-2,0.6);
      \draw[very thick, rounded corners, Red] (-4,0.4) rectangle (-2,-0.8);

      \node (lab1) at (-3,1.6) {$p_1: 1$};      			
      \node (lab1) at (-3,1.2) {$p_2: 0$};      			
      \node (lab2) at (-3,0.8) {$p_3: 4$};

      \node (lab3) at (-3,0.2) {$p_1: 3$};
      \node (lab4) at (-3,-0.2) {$p_2: 3$};  
      \node (lab4) at (-3,-0.6) {$p_3: 2$};

    \end{pgfonlayer}
    \begin{pgfonlayer}{edgelayer}
      \draw[style=inarrow, thick] (1a) to (4a);
      \draw[style=inarrow, thick] (1b) to (2a);
      \draw[style=inarrow, thick] (2a) to (3a);
      \draw[style=inarrow, thick] (3a) to (4a);
      \end{pgfonlayer}
  \end{scope}

  \begin{scope}[xshift=260]
    \begin{pgfonlayer}{nodelayer}
      \node [place,tokens=1, label=above:$p_1$] (1a) at (-1.5,1) {};      			
      \node [place,tokens=1, label=below:$p_2$] (1b) at (-1.5,0) {};
      \node [place,tokens=0, label=right:$p_3$] (3a) at (1.5,0) {};

      \node[transition, label=above:$u_1$] (2a) at (0,0) {};
      \node[transition, label=right:$u_2$] (4a) at (1.5,1) {};

      \draw[very thick, Red] (-1.5,0) circle (0.2);
        \draw[very thick, Red] (-1.7,0) -- (-2,0);
      \draw[very thick, Red] (-1.5,1) circle (0.2);
      \draw[very thick, Red] (-1.7,1) -- (-2,1);

      \draw[very thick, rounded corners, Red] (-4,1.4) rectangle (-2,-0.4);

      \node (lab1) at (-3,1) {$p_1: 4$};      			
      \node (lab1) at (-3,0.5) {$p_2: 3$};
      \node (lab1) at (-3,0) {$p_3: 6$};

    \end{pgfonlayer}
    \begin{pgfonlayer}{edgelayer}
        \draw[style=inarrow, thick] (1a) to (4a);
        \draw[style=inarrow, thick] (1b) to (2a);
        \draw[style=inarrow, thick] (2a) to (3a);
        \draw[style=inarrow, thick] (3a) to (4a);
      \end{pgfonlayer}
  \end{scope}

  \draw[inarrow, thick] (2.5,0.5) --node[above, midway] {Laxator} (4.5,0.5);
\end{tikzpicture}
}
  \end{equation*}
  Transitions $u_1, u_2$ generate the morphisms of the category of executions of the net.
  Looking at our definition, $u_1$ is mapped to the span function that subtracts $p_3$
  from the multiset in input and adds $p_2$ to it. This represents the flow of anti-tokens,
  which goes in the opposite way with respect to the flow of tokens of $u_1$. This again backs up our intuition since anti-places are wired to transitions in the opposite way of their corresponding places.
\end{example}
\subsection{Internalization}
We have defined two different kinds of bound semantics for Petri nets
in the collective-token philosophy, specified by~\cref{def: internal bound semantics common token philosophy}
and~\cref{def: external bound semantics common token philosophy}.
The former is labelled \emph{internal} because places and anti-places are all
put together in the same category, while the latter is labelled \emph{external}
since the information about bounded tokens is encoded in a functor. We now
show that the two approaches are one and the same.
\begin{theorem}\label{thm: equivalence bound semantics collective-token philosophy}
  Let $\NetSem{N}$ be an object of $\PetriCommBound$.
  The category $\CommB{N}$ of~\cref{def: internal bound semantics common token philosophy}
  is isomorphic to the category $\GrothendieckS{N}$ defined as follows:
  \begin{itemize}
    \item Objects of $\GrothendieckS{N}$ are couples $(X,x)$ where
          $X$ is an object of $\CategoryC$ and $x \in \Fun{N}X$.

    \item Morphisms $(X,x) \to (Y,y)$ of $\GrothendieckS{N}$ are couples $(f,s)$
          with $f: X \to Y$ a morphism of $\CategoryC$ and $s \in S_{\Fun{N}f}$
          such that $\Fun{N}f(s) = (x,y)$.
  \end{itemize}
\end{theorem}
\begin{proof}
  Leveraging on the equivalence of categories:
  \begin{equation*}
   \Gamma : \Cat/\Comm{N} \simeq \Cat_l[\Comm{N},\Span]: \Grothendieck{}
  \end{equation*}
  (more details in~\cref{def: total category,def: gamma functor}
  and~\cref{thm: equivalence slice category and functors to span}),
  it is sufficient to consider the counit $\epsilon_N: \CommB{N} \to \Comm{N}$
  of~\cref{prop: comonad collective-token philosophy}, and to notice that $\Gamma{\epsilon_N}$ = $\Fun{N}$.
  Then~\cref{thm: equivalence slice category and functors to span}
  guarantees that $\Grothendieck{\Gamma \epsilon_N} \simeq \epsilon_N$,
  from which the thesis follows.
\end{proof}

\section{Bound semantics in the individual-token philosophy}\label{sec: bound semantics individual}
We now want to generalize the results of~\cref{sec: bound semantics collective}
to the individual-token philosophy. In doing so, we will also prove the worthiness of the categorical approach:
defining external semantics in the individual-token case is considerably harder, but instead of having to do so explicitly,
we will be relying on abstract results from higher category theory.
\begin{definition}[Internal bound semantics -- individual-token philosphy]\label{def: internal bound semantics individual token philosophy}
  Let $N$ be a Petri net, and consider $\Free{N}$, its corresponding FSSMC.
  The \emph{internal bound semantics of $N$ in the individual-token philosophy}
  is given by the FSSMC $\FreeB{N}$ generated by the same information
  of~\cref{def: internal bound semantics common token philosophy}.
\end{definition}
Even if~\cref{def: internal bound semantics common token philosophy}
and~\ref{def: internal bound semantics individual token philosophy} look
nearly identical, things are instantly complicated
by the presence of symmetries. Indeed, proving comonadicity without
resorting to abstract methods becomes very hard, as all
naturality squares require carefully ---and consistently--- selecting symmetries
to commute.
\begin{proposition}\label{prop: comonad individual-token philosophy}
  The assignment $\FreeB{\firstblank}$: $\Free{N} \mapsto \FreeB{N}$ defines a comonad in
  the category of FCMSCs and strict monoidal functors between them, $\FSSMC$.
\end{proposition}
\begin{proof}
  As we will prove in~\cref{thm: equivalence bound semantics individual-token philosophy},
  the category $\FreeB{N}$ fits in a pullback
  \begin{equation*}
    \adjustbox{scale=1}{
\begin{tikzcd}
  \FreeB{N}\pb \ar[r, "Q"]\ar[d, "\epsilon_N"'] & \Span_\bullet\ar[d, "U"] \\
  \Free{N} \ar[r, "{\Fun{N}}"']& \Span
\end{tikzcd}}
    \label{eq: total category pullback}
  \end{equation*}
  We will use the universal property of such pullback to find a comultiplication
  and a counit map and to show the comonad laws; an evident candidate to be a
  counit is the arrow $\epsilon_N : \FreeB{N} \to \Free{N}$; a candidate
  comultiplication map $\alpha_N : \FreeB{N} \Rightarrow \FreeB[2]{N}$
  is obtained from the diagram below as the unique red arrow filling the diagram
  \begin{equation*}
    \adjustbox{scale=1,center}{
\begin{tikzcd}
  \FreeB{N} \ar[dr, red, "\alpha"]\ar[drr, bend left, "Q"]\ar[ddr, bend right, equal]& \\
  & \FreeB[2]{N} \pb \ar[r,"Q_1"]\ar[d,"\epsilon_1"'] & \Span_\bullet \ar[d,"U"]\\
  & \FreeB{N} \ar[r,"{\Fun{N}_1}"'] & \Span
\end{tikzcd}}

  \end{equation*}
  Here, $N_1^\sharp$ is the common
  value of the diagonal $\epsilon_N \Cp {\Fun{N}} = Q \Cp U$ as in~\cref{eq: total category pullback}, and the pair
  $(Q_1,\epsilon_1)$ is obtained iterating the pullback of~\cref{eq: total category pullback}.

  Now that we have the candidate maps (obviously natural in their component $N$,
  because they are defined via a universal property that is functorial in $\Free{N}$),
  we recall the comonad laws in the component $N$:
  \begin{equation*}
    \adjustbox{scale=1}{
\begin{tikzcd}
  \FreeB{N}\ar[r, "\alpha_N"]\ar[d,"\alpha_N"']& \FreeB[2]{N}\ar[d,"\alpha_{\FreeB{N}}"]\\
  \FreeB[2]{N} \ar[r, "\FreeB{\alpha_N}"']& \FreeB[3]{N}
\end{tikzcd}}
    \qquad\qquad
    \adjustbox{scale=1}{
\begin{tikzcd}
  \FreeB[2]{N}\ar[r, "\FreeB{\epsilon_N}"]\ar[d, "\epsilon_{\FreeB{N}}"']\ar[dr,equal] & \FreeB{N} \ar[d, "\alpha_N"]\\
  \FreeB{N}\ar[r, "\alpha_N"'] & \FreeB[2]{N}
\end{tikzcd}}

  \end{equation*}
  The laws are shown as follows:
  \begin{itemize}
    \item The coassociativity requires us to compute the arrows $\FreeB{\alpha_N}$ and $\alpha_{\FreeB{N}}$:
    First, we build the definition of $\FreeB[3]{N}$, as the upper left corner in the diagram
    \begin{equation*}
      \adjustbox{scale=1,center}{
  \begin{tikzcd}
    {\FreeB[3]{N}}
    \ar[r, "Q_2"]
    \ar[d, "\epsilon_2", shift left=.5em]& {\Span_\bullet}
    \ar[d, equal] \\
    {\FreeB[2]{N}} \ar[u, shift left=.5em, "\alpha_{\FreeB{N}}"]
    \ar[r, "Q_1"]
    \ar[d, "\epsilon_1", shift left=.5em]& {\Span_\bullet}
    \ar[d, equal]\\
    {\FreeB{N}} \ar[u, shift left=.5em, "\alpha_N"]
    \ar[r, "Q"']
    \ar[d,"\epsilon_N"']& {\Span_\bullet}
    \ar[d,"U"] \\
    {\Free{N}}
    \ar[r, "{\Fun{N}}"'] & {\Span}\\
  \end{tikzcd}
}

    \end{equation*}

    \vspace{-1.5em}
    Now, according to the definitions, $\alpha_{\FreeB{N}}$ is the arrow working as a right inverse of $\epsilon_2$, whereas $\FreeB{\alpha_N}$ results from the diagram
    \begin{equation*}
      \adjustbox{scale=1,center}{
  \begin{tikzcd}
    \FreeB{N} \arrow[rddd, equal, bend right] \ar[rd, "\alpha_N" description] \ar[rr] && \FreeB[3]{N} \ar[rr, "Q_2"] \ar[dd, "\epsilon_2"', pos=.25] && \Span_\bullet \ar[dd, "U"] \\
    & \FreeB[2]{N} \ar[rr, "Q_1"', pos=.25] \ar[dd, "\epsilon_1"'] \ar[ur,red, "\color{red}\FreeB{\alpha_N}"] && \Span_\bullet \ar[dd, "U", pos=.25] \ar[ru, equal]\\
    && \FreeB[2]{N} \ar[rr, "{\Fun{N}}_2"', pos=.25]&& \Span\\
    & \FreeB{N} \ar[rr, "{\Fun{N}}_1"'] \ar[ru, "\alpha_N"']&& \Span \ar[ru, equal]
  \end{tikzcd}
}

    \end{equation*}
    The coassociativity property is implied from the fact that
    \begin{equation}
      \begin{cases}
        \alpha_N \Cp \alpha_{\FreeB{N}} \Cp \epsilon_2 = \alpha_N \Cp \FreeB{\alpha_N} \Cp\epsilon_2\\
        \alpha_N \Cp \alpha_{\FreeB{N}} \Cp Q_2 = \alpha_N \Cp  \FreeB{\alpha_N} \Cp Q_2
      \end{cases}
      \label{eq: comonad comultiplication law individual-token philosophy}
    \end{equation}
    since the projection maps $(\epsilon_2, Q)$ from a pullback are jointly monic (as always when computing a limit).
    \cref{eq: comonad comultiplication law individual-token philosophy}
    holds because of the commutative conditions found so far:
    \begin{align*}
      \alpha_N \Cp \alpha_{\FreeB{N}} \Cp \epsilon_2 & = \alpha_N \Cp \Id{\FreeB[2]{N}}\\
                                                                                      & = \Id{\FreeB{N}} \Cp \alpha_N\\
                                                                                      & = \alpha_N \Cp \epsilon_1 \Cp \alpha_N \\
                                                                                      & = \alpha_N \Cp \FreeB{\alpha_N} \Cp\epsilon_2\\
      \alpha_N \Cp \alpha_{\FreeB{N}} \Cp Q_2          & =  \alpha_N \Cp \alpha_{\FreeB{N}} \Cp Q_2 \Cp \Id{\Span_\bullet}\\
                                                                                      & = \alpha_N \Cp \alpha_{\FreeB{N}} \Cp \epsilon_2 \Cp Q_1\\
                                                                                      & = \alpha_N \Cp \Id{\FreeB[2]{N}} \Cp Q_1\\
                                                                                      & = \alpha_N \Cp Q_1 \Cp \Id{\Span_\bullet}\\
                                                                                      & = \alpha_N \Cp  \FreeB{\alpha_N} \Cp Q_2
    \end{align*}
    \item The left and right counit laws follow from a similar chain of reasoning, and all in all from the definition of the counit and comultiplication. \qedhere
  \end{itemize}
\end{proof}
As for the commutative case, we now want to provide a semantics that is the external
counterpart of~\cref{def: internal bound semantics individual token philosophy}.
\cref{def: non-local semantics collective-token philosphy} can be ported almost verbatim:
\begin{definition}[Non-local semantics -- individual-token philosophy]\label{def: non-local semantics individual-token philosphy}
  Let $N$ be a Petri net and let $\Semantics$ be a monoidal bicategory.
  A \emph{Petri net with a non-local semantics} is a couple $\NetSem{N}$ with
  $\Fun{N}$ a \emph{lax-monoidal-lax functor} $\Free{N} \to \Semantics$.

  A morphism $\NetSem{N} \to \NetSem{M}$ of Petri nets
  with non-local semantics is a strict monoidal functor $\Free{N} \xrightarrow{F} \Free{M}$
        making the obvious triangle commute.
  Overloading notation, we denote the category of Petri nets with a non-local semantics with $\PetriS{\Semantics}$.
\end{definition}
Yet, actually producing a lax-monoidal-lax functor $\Free{N} \to \Span$
akin to the one in~\cref{def: external bound semantics common token philosophy}
is not as easy:
As pointed out in~\cref{ex: net executions difference in causality},
in the commutative case, it does not really matter which object we apply a given morphism to,
but only how many objects of some sort are turned into objects of some other sort. This allowed us
to represent morphisms between anti-places using spans with the preimage of $\chi$ in their tips,
and kept things manageable at the level of intuition.

When non-trivial symmetries are around, the causal relationship between
morphisms becomes meaningful: for instance, $f \Cp g$ and $f \otimes g$ will, in general, act very differently on objects. Moreover, looking at the fibre of $\chi$ over a morphism generator $f$, we can describe it
quite explicitly as the set of all (possibly empty) stackings of string diagrams
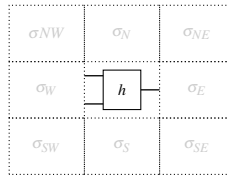
\begin{equation*}
  \begin{tikzpicture}[yscale=.75]
  \twoAr{h}
  \step{\akasa[\sigma_E]}
  \step[-1]{\akasa[\sigma_W]}
  \up[4]{
    \step{\akasa[\sigma_{NE}]}
    \akasa[\sigma_N]
    \step[-1]{\akasa[\sigma{NW}]}
  }
  \down[4]{
    \step{\akasa[\sigma_{SE}]}
    \akasa[\sigma_S]
    \step[-1]{\akasa[\sigma_{SW}]}
  }
\end{tikzpicture}
\end{equation*}
where various symmetries $\sigma_W,\sigma_E, \sigma_{NE},\dots$ ``surround'' $f$
on each side. If $f$ is an identity, say over an object $A$, then $\chi^{-1}(1_A)$
coincides with the set of all symmetries of $\Free{N}$, making it clear that a
lax-monoidal-lax semantics for the individual-token case will not preserve identities strictly.
\subsection{External semantics via abstract machinery}\label{sec: external semantics via abstract machinery}
With such a great deal of technicality,
generalizing~\cref{def: internal bound semantics common token philosophy} ``manually''
can easily go wrong. As we did already in~\cite{Genovese2021a}, we solve the problem by inverting our reasoning:
We will use abstract machinery to obtain a lax-monoidal-lax functor from the counit $\FreeB{N} \to \Free{N}$
of the comonad in~\cref{prop: comonad individual-token philosophy}, in such a way that
the isomorphism with $\FreeB{N}$ will be guaranteed after internalizing.
\begin{definition}[Total category \cite{Pavlovic1997}]\label{def: total category}
  Let $\CategoryC$ be a 1-category, regarded as a locally discrete bicategory, and let $F: \CategoryC \to \Span$ be a lax-monoidal-lax functor.
  The \emph{total category} of $F$ is the 1-category $\Grothendieck{F}$ defined as the pullback (in $\Cat$)
  \begin{equation*}
    \begin{tikzcd}
  \Grothendieck{F} \ar[r]\ar[d, "\epsilon_F"'] \pb & \Span_\bullet\ar[d, "U"] \\
  \CategoryC \ar[r, "F"'] & \Span
\end{tikzcd}

  \end{equation*}
  where $\Span_\bullet$ is the bicategory of spans between pointed sets,
  and $U$ is the forgetful functor.
  More concretely, $\Grothendieck{F}$ is defined as the category\footnote{
    All 2-cells are identities, due to the 2-discreteness of $\CategoryC$.
  } where
  \begin{itemize}
    \item 0-cells of $\Grothendieck{F}$ are couples $(X,x)$ where
          $X$ is a 0-cell of $\CategoryC$ and $x \in FX$;
    \item 1-cells $(X,x) \to (Y,y)$ of $\Grothendieck{F}$ are couples $(f,s)$
          where $f: X \to Y$ is a 1-cell of $\CategoryC$ and $s \in S_{Ff}$ with
          $Ff(s) = (x,y)$. 
        \end{itemize}
        Representing a span as a function $(S,s) \to (X\times Y, (x,y))$ between (pointed) sets,
        a morphism is a pair $(f,s)$ such that $Ff : s\mapsto (x,y)$.
\end{definition}
\begin{definition}\label{def: gamma functor}
  Let $\CategoryC$ and $\CategoryD$ be 1-categories, and let $F: \CategoryD \to \CategoryC$ be a functor.
  There is a lax functor $\Gamma F: \CategoryC \to \Span$, with $\CategoryC$ regarded as a locally discrete bicategory,
  defined as follows:
  \begin{itemize}
    \item On objects, $C \in \CategoryC$ is mapped to the set $\Suchthat{D \in \CategoryD}{FD = C}$.
    \item On morphisms, $C \xrightarrow{f} C'$ in $\CategoryC$ is mapped to the span:
    \begin{equation*}
      \Suchthat{D \in \CategoryD}{FD = C} \xleftarrow{dom} \Suchthat{g \in \CategoryD}{Fg = f} \xrightarrow{cod} \Suchthat{D \in \CategoryD}{FD = C'}
    \end{equation*}
  \end{itemize}
\end{definition}
The following fact was first observed by J. Bénabou in \cite{Benabou1967} (see also \cite[Th.\@\xspace 5.4.5]{Loregian2020}), where instead of $\Span$ there is the bicategory of profunctors:
\begin{theorem}[\cite{Pavlovic1997}]\label{thm: equivalence slice category and functors to span}
  For any category $\CategoryC$, there is an equivalence of categories
  \begin{equation*}
    \Gamma: \Cat / \CategoryC \simeq \Cat_l [\CategoryC, \Span]: \Grothendieck{}
  \end{equation*}
  where the right hand side is the 1-category of lax monoidal functors and functional (i.e.\@\xspace whose components are all functions) natural transformations,
  and $\Grothendieck{}$, $\Gamma$ are defined as in~\cref{def: total category}
  and~\cref{def: gamma functor}, respectively. In particular, for $F : \CategoryD \to \CategoryC$,
  $\Grothendieck{\Gamma F}$ is isomorphic to $\CategoryD$.
\end{theorem}
Relying upon \cref{def: gamma functor}, we are ready to define our external semantics.
\begin{definition}[External bound semantics -- individual-token philosphy]\label{def: external bound semantics individual token philosophy}
  We define a lax-monoidal-lax functor $\Fun{N}: \Free{N} \to \Span$ as $\Gamma \epsilon_N$,
  where $\epsilon_N : \FreeB{N} \to  \Free{N}$ is the functor sending any anti-place to the monoidal
  unit, and generating morphisms to themselves.
\end{definition}
\begin{proposition}
  The functor of~\cref{def: external bound semantics individual token philosophy} is lax-monoidal-lax.
\end{proposition}
\begin{proof}
  That $\Fun{N}$ is lax is true by definition. We need to prove that it laxly preserves the monoidal structure. This is readily done by noticing that
  $\Fun{N}$ sends all objects to the set of multisets on the set of object generators of $\FreeB{N}$
  which are of the form $\Bwd{A}$. This set is isomorphic to $\Str{S}$, the set of objects of $\Free{N}$,
  and we obtain an obvious laxator $\Str{S} \times \Str{S} \to \Str{S}$ by concatenating strings.

  The laxator diagram commutes only up to a 2-cell, which is to be expected given that we are working with lax functors. This 2-cell is readily obtained from the following inclusion
  \begin{equation*}
    \Suchthat{h \in \CategoryD}{Fh = f \otimes g} \subseteq  \Suchthat{h \in \CategoryD}{Fh = f} \times  \Suchthat{h \in \CategoryD}{Fh = g}
  \end{equation*}
  between the span tips of $\Fun{N}(f \otimes g)$ and $\Fun{N}f \times \Fun{N}g$, respectively.
\end{proof}
\begin{theorem}\label{thm: equivalence bound semantics individual-token philosophy}
  Let $\NetSem{N}$ be an object of $\PetriFreeBound$.
  The category $\FreeB{N}$ of~\cref{def: internal bound semantics common token philosophy} is isomorphic to the category resulting from the following pullback in $\Cat$:
  \begin{equation*}
    \begin{tikzcd}
  \GrothendieckS{N} \ar[r]\ar[d, "\epsilon_{\Fun{N}}"'] \pb & \Span_\bullet\ar[d, "U"] \\
  \Free{N} \ar[r, "\Fun{N}"'] & \Span
\end{tikzcd}

  \end{equation*}
  where $\Span_\bullet$ is the bicategory of spans between pointed sets,
  and $U$ is the forgetful functor. Explicitly:
  \begin{itemize}
    \item Objects of $\GrothendieckS{N}$ are couples $(X,x)$ where
          $X$ is an object of $\CategoryC$ and $x \in \Fun{N}X$.

    \item Morphisms $(X,x) \to (Y,y)$ of $\GrothendieckS{N}$ are couples $(f,s)$
          with $f: X \to Y$ a morphism of $\CategoryC$ and $s \in S_{\Fun{N}f}$
          such that $\Fun{N}f(s) = (x,y)$.
  \end{itemize}
\end{theorem}
\begin{proof}
  Unrolling definitions, this is just~\cref{thm: equivalence slice category and functors to span}.
\end{proof}

\section{Conclusion and future work}\label{sec: conclusion and future work}
In this work, we proved how the functorial semantics approach developed
in~\cite{Genovese2020} to model guarded nets can be generalized to include
other traditionally familiar constructions such as bounding. To achieve such a result,
we refined a technique already developed in~\cite{Genovese2021a} by endowing nets with a
\emph{lax-monoidal-lax} functorial semantics. This improvement is a conceptually relevant
leap: if in~\cite{Genovese2020} strong monoidal functors endowed tokens with properties that
could be only locally true, in~\cite{Genovese2021a} and here tokens can be endowed with properties
about the global state of the net (such as distributions of tokens in
other places). The laxity requirement, which translates into the possibility to consider ensembles of tokens as a single thing, builds a tension between the intrinsic local nature of tokens and their global properties.

Mathematically, our main contribution consists in proving how the formalism of functorial semantics is mature enough to scale to complicated situations, such as non-local semantics, in the presence of symmetries; this requires adopting
a higher-categorical point of view, and resorting to the entirely abstract description
of many constructions that we were able to implement naively in the collective-token
case.

As for directions for future work, we want to investigate what else our formalism
can cover, hierarchic nets~\cite{Jensen2009} and nets
with inhibitor arcs~\cite{Zaitsev2013, Zaitsev2012} being the most promising candidates. In general,
we are now convinced of the following fact, which we consider a
significative conceptual contribution:
\begin{center}
  \emph{
    Studying the categorical semantics of extensions of Petri nets amounts
    to classifying lax functors $\CategoryC \to \Span$ from a
    free (commutative, symmetric) monoidal category $\CategoryC$.
  }
\end{center}

\printbibliography
\typeout{get arXiv to do 4 passes: Label(s) may have changed. Rerun}
\end{document}